\documentclass[a4paper,12pt]{article}

\usepackage[utf8]{inputenc}
\usepackage{amsfonts}
\usepackage{amsmath,latexsym,amssymb}
\usepackage{amsthm} %inma para entorno remark
\usepackage{rotating}
\usepackage{multirow,longtable}
\usepackage{tabularx}
\usepackage{colortbl}
\usepackage{dcolumn}
\usepackage{booktabs}
\usepackage{graphicx}
\usepackage{float}
\usepackage{boldline,multirow}
\usepackage{subfigure}
\usepackage{tikz}
\usepackage{lscape}
\usepackage{verbatim}
\usepackage[backend=biber,style=apa]{biblatex}
\bibliography{references}
\usepackage{sectsty}
\allsectionsfont{\normalfont\mdseries}
\usepackage{comment}

\newtheoremstyle{mystyle}%                % Name
{}%                                     % Space above
{}%                                     % Space below
{\itshape}%                             % Body font
{}%                                     % Indent amount
{\bfseries}%                            % Theorem head font
{.}%                                    % Punctuation after theorem head
{ }%                                    % Space after theorem head, ' ', or \newline
{}%
\theoremstyle{mystyle}

\newtheorem{propo}{Proposition}[section]
\newtheorem{lemma}{Lemma}[section]

\usepackage[linesnumbered,ruled,vlined]{algorithm2e}
\SetKwRepeat{Repeat}{repeat at each node of the branching tree}{until}

\usepackage[normalem]{ulem}

\newcommand{\msout}[1]{\text{\sout{\ensuremath{#1}}}}

\definecolor{armygreen2}{rgb}{0.99, 0.53, 0.93}

\definecolor{armygreen3}{rgb}{0.53, 0.53, 0.99}

\definecolor{armygreen4}{rgb}{0.023, 0.93, 0.73}

\newcommand{\AMR}[2]{{\color{armygreen2}#1 \ifmmode\msout{#2}\else\sout{#2}\fi}}

\newcommand{\IEMM}[2]{{\color{blue}#1
\ifmmode\msout{#2}\else\sout{#2}\fi}}

\begin{document}

\title{Hub location problems with asymmetric allocation}

%\setcitestyle{authoryear,open={(},close={)}}

 \author{Inmaculada Espejo$^a$, Mercedes Landete$^b$, Marina Leal$^b$ and Alfredo Mar{\'i}n$^c$}
	
 \date{\small
 $^a$ Departamento de Estad\'istica e Investigaci\' on Operativa, Universidad de C\'adiz, Spain.\\
 $^b$ Centro de Investigaci\' on Operativa, Universidad Miguel Hern\'andez, Spain.\\
 $^c$ Departamento de Estad\'istica e Investigaci\' on Operativa, Universidad de Murcia, Spain.\\
 }

\maketitle \thispagestyle{empty}

\begin{abstract}
Hub location problems are central to optimizing logistics, telecommunications, and transportation networks
by consolidating flows through strategically placed hubs. While existing models assume symmetric allocation
--where hubs handle incoming and outgoing flows uniformly-- real-world applications often require asymmetric
handling of origins and destinations. This asymmetry is better suited to model the complex logistical trade-off between cost efficiency at the origin and service resilience at the destination. This paper introduces the Asymmetric Hub Location Problem ($(r,s)$-AHLP),
a novel framework where origins and destinations may connect to hubs under distinct allocation limits
($r$ and $s$, respectively). This framework generalizes and includes the existing allocation structures in the literature. We then focus on the $(1,p)$-AHLP variant, where origins are single-assigned and
destinations are multi-assigned, motivated by applications in humanitarian logistics and global supply chains
(e.g., UN relief networks, e-commerce fulfillment).

We propose two integer programming formulations for the $(r,s)$-AHLP: A four-index adaptation of classical models and a new three-index formulation. 
Given the applications of the $(1,p)$-AHLP variant, we focus on develop a solution procedure based on a new compact formulation, which serves as the base for a branch-and-cut algorithm. The solution methodology is strengthened by novel valid inequalities derived using Farkas' Lemma. We propose a separation algorithm for the valid inequalities that results in the decomposition of the associated optimization problem into the dual of classical transportation problems, ensuring fast and efficient cut generation within the branch-and-cut framework.
The computational study, performed on standard datasets commonly used in hub location literature, demonstrates the high effectiveness and efficiency of the proposed solution methodology. By explicitly addressing asymmetric flow, our work significantly advances hub location theory and provides a starting point for complex multi-level network design studies.
\end{abstract}

\textbf{Keywords:} Hub location, asymmetric allocation, integer programming, humanitarian logistics, supply chain optimization.

\section{Introduction}

Hub Location Problems (HLP) arise in various applications, such as telecommunications and transportation systems, where multiple origin-destination pairs exchange goods or data. Instead of connecting each pair directly, transshipment points (hubs) consolidate flows from origins and distribute them to destinations. This approach exploits economies of scale by routing products through centralized hubs.
Several reviews, including \textcite{Alumur2008}, \textcite{Campbell2012}, \textcite{Farahani2013}, \textcite{Contreras2019} and \textcite{Alumur2021},
highlight the breadth of research and applications in this field.

In general, hub location problems involve deciding where to locate hub facilities and how to assign origin–destination pairs to them, ensuring that routes pass only through established hubs. Different allocation strategies define problem variants: (i) single allocation,  where each site is assigned to exactly one hub;
(ii) multiple allocation,  where sites may connect to several hubs; and (iii) $r$-allocation,  an intermediate strategy in which each site connects to at most $r$ hubs.

Hub location problems can also be classified according to the number of hubs to be opened:
(i) $p$-hub median problems,  where a maximum number $p$ of hubs can be located; and (ii) fixed-cost hub location problems, where hubs can be opened at a given cost without a predetermined limit on their number.

Additionally, problems may be capacitated (hubs have limited capacity) or uncapacitated, and several objectives can be considered, typically of the median or center types.

However, the vast majority of existing hub location models rely on the critical assumption of symmetric allocation, meaning that the rules and limits for connecting origins to hubs are the same as those for connecting destinations to hubs. This symmetry fails to capture the complexity of modern logistics networks where consolidation and distribution phases operate under fundamentally different mandates. For instance, in global supply chains and e-commerce fulfillment, the primary goal for flow consolidation at the origin is often efficiency and cost minimization (typically favoring single assignment), whereas the goal for final delivery at the destination is speed, service quality, and resilience against disruption (necessitating multiple allocation pathways).

%This paper focuses on hub location problems with a novel allocation strategy related to both single and multiple allocation models. We propose what we call the asymmetric hub location problem, in which origins can send their products to a subset of hubs, and destinations can receive them from another (potentially different) subset. 

To address this theoretical and practical gap, this paper introduces the {\it Asymmetric Hub Location Problem} (AHLP), a novel framework where origins and destinations may connect to hubs under distinct allocation limits,  generalizing and encompassing all existing allocation paradigms in the literature as special cases.

We strategically focus our methodological developments on the specific case where origins are constrained to single assignment, and destinations are allowed multiple assignment. This is the most logistically relevant scenario for simultaneously optimizing cost efficiency and network resilience, as evidenced by applications detailed in Section \ref{applications} in global e-commerce fulfillment (e.g., Amazon, FedEx) and humanitarian relief efforts (e.g., UNHRD, IFRC).

The specific structure of this particular version makes it non-trivial to solve using existing methods. We will show with examples throughout computational experiments that this variant produces structurally different optimal network designs compared to symmetric models existing in the literature.

This allocation structure, while highly common in modern distribution networks, is virtually new to hub location theory. To the best of our knowledge, the only related work is found within the highly specialized context of the ordered hub location problem \parencite{orderedhub2011, orderedhub2013, orderedhub2016}. In these studies, a flow structure of single origin hub to multiple destination hubs was utilized, but this structure emerged as a consequence of enforcing the flow sequence required by that specific problem, not as a general, dedicated allocation strategy. We introduce a truly asymmetric approach where allocation limits are dictated independently by the distinct logistical roles and requirements of the origin and destination, respectively. We propose new mathematical models and present computational results showing the impact of asymmetry on optimal allocation.

%As we will see in Section \ref{applications}, there is a wealth of applications where the assumption of symmetry does not fit the real situation. In other words, symmetric models can be seen as unnecesarily limiting. The aim of this article is to fill the gap in the literature.

\subsection{Literature review on allocation strategies and modeling approaches}

To contextualize our proposed models, we briefly review the relevant literature on integer programming formulations for the uncapacitated median hub location problem and related assignment problems.  While our models align with the uncapacitated $p$-hub median framework, they can be adapted to fixed-cost or capacitated settings with minor modifications.

\bigskip

\noindent \textbf{Multiple allocation}

\bigskip
Multiple allocation strategies, where origin-destination pairs may use different hub combinations, are generally easier to formulate and computationally less demanding than their single-allocation counterparts.

Early integer programming formulations for $p$-hub and fixed-cost problems used four-indexed variables \parencite{Campbell1992, Campbell1994}. These initial models provided weak lower bounds, but improvements followed in \textcite{OKelly1996} and \textcite{SkorinKapov1996}. Later, \textcite{Marin2006} and \textcite{Hamacher2004} refined the four-index formulations, while \textcite{Marin2005} and \textcite{Canovas2007} advanced solution methods. Large-scale instances were later tackled via Benders decomposition \parencite{Contreras2011}, and more recently, \textcite{LANDETE2024EJOR} introduced a large new family of valid inequalities for branch-and-cut procedures.

Alternative formulations with three-indexed variables were introduced in \textcite{EK1998b, EK1998a}. While four-index formulations yield strong bounds, they require significant computational resources. Three-index models are more compact but provide weaker bounds. To balance these trade-offs, \textcite{Garcia2012} developed a two-index formulation enhanced with valid inequalities, enabling the solution of large instances.

\bigskip

\noindent \textbf{Single allocation}

\bigskip

Single allocation problems, where each site connects to only one hub, are known to be computationally more complex. Even with fixed hub locations, the allocation subproblem remains NP-hard \parencite{Kara1999}.

The first formulation, proposed in \textcite{OKelly1987}, used quadratic terms. Later, \textcite{Campbell1994} introduced a linear integer programming formulation. \textcite{SkorinKapov1996} then developed a path-based mixed-integer formulation with fewer variables, widely adopted in decomposition methods.

To reduce the number of variables, projection techniques were applied to the path-based formulation \parencite{Labbe2004, Labbe2005, Camargo2011, Camargo2012}. \textcite{EK1996} proposed a flow-based formulation, which further reduced model size while maintaining effectiveness. This approach was extended to capacitated and balanced problems \parencite{Correia2011}, as well as to budget-constrained network upgrade problems \parencite{LANDETE2024TRB}.

A new compact formulation proposed in \textcite{Ebery2001} required far fewer variables than previous models, with a similar number of constraints. However, the model in \textcite{SkorinKapov1996} provided tighter linear relaxations, and the formulation in  \textcite{EK1996} was computationally faster. Recent work includes linearization techniques for quadratic terms \parencite{Meier2016, GKara2019, Rostami2019}, Benders decomposition for large-scale instances \parencite{GKara2019}, and branch-and-cut algorithms \parencite{Rostami2019}.

A recent compact formulation with fewer variables was proposed in \textcite{sahlp}, along with valid inequalities and a branch-and-cut algorithm. This method efficiently solved large-scale instances in competitive times.

\bigskip

\noindent \textbf{$r$-allocation}

\bigskip

The $r$-allocation strategy generalizes both single and multiple allocation models, allowing each site to be connected to at most $r$ hubs. When $r=1$, it reduces to the single-allocation case, while for large values of $r$, it approaches multiple allocation.

\textcite{Yaman2011} developed mixed-integer formulations and found that single allocation solutions lead to higher  costs compared to multiple allocation, but significant savings arise when $r=2$ or $3$. \textcite{Corberan2019} further strengthened the model by introducing valid inequalities, optimality cuts, and a branch-and-cut algorithm, supported by polyhedral analysis and separation procedures.

\bigskip

\noindent\textbf{Quadratic assignment and semi-assignment}

\bigskip

Given two matrices $A$ and $C$, the Quadratic Assignment Problem (QAP) seeks a simultaneous row and column permutation of $C$ that minimizes the discrepancy between $A$ and the permuted version of $C$. This problem is among the most challenging in combinatorial optimization \parencite{5,6,32,34,38}, with applications in
graph matching \parencite{3,13}, de-anonymization and privacy \parencite{43}, protein network alignment \parencite{47}
and traveling salesman problems \parencite{16,29}, to cite a few.

A common approach to the QAP is convex relaxation \parencite{51}. Notably, \textcite{27} introduced a linear programming relaxation and derived lower bounds, which were later refined by \textcite{34}. Further algorithmic advances are surveyed in \textcite{5,6,34,38}.

Closer to the problems we introduce in this work is the Quadratic Semi-Assignment Problem (QSAP, \textcite{saito2009}),
where assignments are one-directional (unlike QAP’s two-way permutations). Polynomially solvable cases and lower bounds were studied
in \textcite{a21,a22,a23}, with applications in scheduling (e.g., unrelated parallel machines \textcite{a31}),
simplified facility location \parencite{a28} and metric labeling \parencite{a16}.

The inclusion of these two Assignment Problems (AP) helps establish the inherent difficulty of the allocation decision in problems involving flow and network structure. Specifically, the structural complexity of QSAP in modeling unidirectional assignments is conceptually related to the allocation difficulties introduced by the asymmetric constraints in the AHLP, which highlights the necessity of a computationally efficient and effective modeling approach.

\subsection{Contributions of the paper}

This paper’s main contribution is the development of a novel and comprehensive framework for hub network design, defined by the general Asymmetric Hub Location Problem (AHLP). This framework explicitly models the strategic need for distinct allocation rules at origins and destinations. It generalizes and includes all previously studied allocation structures (single, multiple, and $r$-allocation) found in the literature. We address this practical and theoretical need through the following key contributions:
\begin{itemize}
\item We introduce and formulate the general Asymmetric Hub Location Problem that allows modeling the strategic need for distinct allocation criteria at the origin and the destination. We then explicitly develop a solution methodology for the particular case of single allocation at the origin and multiple allocation at the destination, a variant of high logistical relevance that captures the conflict between cost control and service resilience, and which is highly applicable in global supply chains and humanitarian logistics.
\item We propose two general Integer Programming formulations for the AHLP framework: (i) an adaptation of the classical four-index model; and (ii) a new three-index formulation. Both formulations are valid for a general cost structure, which significantly broadens the applicability of our work compared to many existing models in the literature that are restricted to the disaggregated cost structure. For the variant single allocation at the origin and multiple allocation at the destination,  we propose a compact formulation utilizing aggregated cost variables. 
\item We develop a solution procedure based on a branch-and-cut framework. We derive new valid inequalities using Farkas' Lemma applied to an auxiliary formulation, which significantly tightens the linear programming relaxation of our compact model. Furthermore, we demonstrate the computational efficiency of the cutting plane generation by proving that the associated separation problem decomposes into a set of independent subproblems, each of which is the dual of a classical transportation problem, allowing for their efficient solution via dedicated algorithms.
\item We perform a thorough computational study to evaluate the proposed methodology. Specifically, we quantify the operational impact and the cost trade-offs associated with imposing asymmetric allocation constraints, which offer clear guidance on optimizing network efficiency and resilience simultaneously.
\item We significantly advance hub location theory by explicitly addressing asymmetric flow, and provide a starting point for future complex multi-level network design studies.
\end{itemize}

The rest of the paper is organized as follows. 
Section 2 provides a detailed description of the asymmetric hub location problem. Section 3 discusses several relevant applications of this asymmetric problem in the specific case where single allocation is considered for origins and multiple allocations for destinations. Section 4 is the most technical one, where we propose several optimization models to solve the general asymmetric problem: the first model is quadratic, the second is linear and employs four-index variables, and the third is also linear but uses three-index variables. In Section 5, the general models are specialized to the particular case of single allocation for origins and multiple allocation for destinations, and two families of valid inequalities are proposed, one of them derived from Farkas’ Lemma. Section 6 presents the computational results obtained for several instances from the well-known AP library. From these computational experiments, we identify the most efficient branch-and-cut algorithm for our model and analyze the impact of asymmetric permissibility on the solution. 
Finally, Section 7 summarizes our conclusions.

\section{A new paradigm: the asymmetric hub location problem}

The classical paradigm of hub location problems involves a set of sites $N$ (cardinality $n$)
and a set of potential hubs $H$ (cardinality $h$), with a non-negative cost $C_{ijkm}$ for each quadruple
$(i,j,k,m)$ in $N\times N\times H\times H$. Given parameters $2\le p\le h-1$ (maximum number of hubs) and $1\le r\le p$
(allocation limit), the more general {\em $r$-Allocation Hub Location Problem} ($r$-AHLP) selects
\begin{itemize}
 \item a non-empty set $P\subset H$, $|P|\le p$, called set of hubs and
 \item for every $i\in N$, a set $P(i)\subseteq P$ with $1\le |P(i)|\le r$
\end{itemize}
to minimize the cost
\begin{equation}
 \sum_{i\in N} \sum_{j\in N} \min_{k\in P(i),m\in P(j)} C_{ijkm}.  \label{f.o1}
\end{equation}

Special cases include the Multiple Allocation Hub Location Problem (MAHLP, $r=p$, $N=H$) and the
Single Allocation Hub Location Problem (SAHLP, $r=1$, $N=H$).
A common framework assumes as input
\begin{itemize}
 \item $(w_{ij})$ a non-negative matrix representing the amount of product to be sent from $i\in N$ to $j\in N$,
 \item $(c_{ij})$ a non-negative matrix with zeros in the diagonal representing the cost of sending a unit
 of product directly from $i$ to $j$, and possibly satisfying the triangle inequality,
 \item a decomposed structure $C_{ijkm}=w_{ij}(\gamma c_{ik}+\alpha c_{km}+\beta c_{mj})$ with $0\le \alpha\le \beta ,\gamma$.
\end{itemize}
This is a particular case, here termed {\em Dissagregated Hub Location Problem}.

Fixed costs $\sum_{k\in P} f_k$ with $f_k\ge 0$ $\forall k\in H$ can be added to the objective function \eqref{f.o1},
typically setting $p=h$ to balance hub-opening and routing $C$-costs.

Existing models are symmetric, meaning hubs handle both incoming and outgoing connections identically. This assumption lacks
flexibility for real-world applications. 
We propose an asymmetric framework with distinct origin ($O$), destination ($D$), and hub ($H$) sets
(cardinalities $o$, $d$, $h$), independent allocation limits $r$ (outgoing) and $s$ (incoming), and a general cost
structure $C_{ijkm}$, $\forall (i,j,k,m)\in O\times D\times H\times H$.

The first, more general problem we deal with, named {\em $(r,s)$-Allocation Hub Location Problem} ($(r,s)$-AHLP) is to choose
\begin{itemize}
 \item a non-empty set $P\subset H$, $|P|\le p$, called set of hubs,
 \item for every $i\in O$, a non-empty set $P^o(i)\subseteq P$ with $1\le |P^o(i)|\le r$,
 \item and, for every $j\in D$, a non-empty set $P^d(j)\subseteq P$ with $1\le |P^d(j)|\le s$
\end{itemize}
to minimize
\begin{equation}
 \sum_{i\in O} \sum_{j\in D} \min_{k\in P^o(i),m\in P^d(j)} C_{ijkm}.  \label{f.o}
\end{equation}

From now on, superindices $o$ and $d$ are simply letters to distinguish ``origin'' and
``destination''. On the contrary, subindices will run over sets of integers.

With this notation, the MAHLP is the particular case $(p,p)$-AHLP when $O=D=H$. Regarding the SAHLP and the
$r$-AHLP, they can be considered two particular cases of $(r,s)$-AHLP with the additional constraints $P^o(i)=P^d(i)$ $\forall i\in O$.
They correspond, respectively, with the cases $(r,s)=(1,1)$ and $s=r$ when $O=D=H$.

The general framework established by the $(r,s)$-AHLP, along with the precise conditions necessary to recover existing models from the literature, is summarized in Table \ref{table:par}.
\begin{table}[htbp]
    \centering
    \small \renewcommand{\arraystretch}{1.75}
    \renewcommand{\multirowsetup}{\centering}
    \begin{tabular}{ccccc}
        \toprule \toprule
        \textbf{Model} & \textbf{Origin} & \textbf{Destination} & \textbf{Asumptions}& \textbf{Status}\\
        \midrule \midrule
        $(r,s)$-AHLP & $r \ge 1$ & $s \ge 1$ & None &New\\ \midrule
       $(1,p)$-AHLP & $r=1$ & $s=p$ & None& New\\ \midrule
       MAHLP& $r=p$ & $s=p$ & $O=D=H$ & Classical HLP \\ \midrule
      \multirow{2}{24mm}{$r$-AHLP}& \multirow{2}{24mm}{$r$} & \multirow{2}{24mm}{$r$} & $O=D=H$&\multirow{2}{24mm}{Classical HLP}\\&&& $P^o(i) = P^d(i)$, $\forall i$ \\ \midrule
     \multirow{2}{24mm}{SAHLP}&\multirow{2}{24mm}{$r=1$} &\multirow{2}{24mm}{$s=1$} & $O=D=H$&\multirow{2}{24mm}{Classical HLP}\\ &&&$P^o(i) = P^d(i)$, $\forall i$ \\ \midrule
    &&&$O=D=H$\\
    QSAP&$r=1$ &$s=1$& $p=h$& Classical AP\\&&&\mbox{allocation in O/D to same hub}&\\ 
        \bottomrule \bottomrule
    \end{tabular}\caption{Relationship between $(r,s)$-AHLP and classical  problems}
    \label{table:par}
\end{table}

Although we will build integer programming formulations for $(r,s)$-AHLP, we will pay special attention to $(1,p)$-AHLP, i.e., with single allocation of the origins and multiple allocation of the destinations (equivalent to the reverse problem $(p,1)$-AHLP). As explained in the next section, our motivation is to analyze this case of particular importance in the field of humanitarian logistic, as well as in
other scopes of activity. For this reason we particularize in the following to specifically define the
{\em $(1,p)$-Allocation Hub Location Problem}: Select
\begin{itemize}
 \item a non-empty set of hubs $P\subset H$, $|P|\le p$,
 \item for every $i\in O$ a hub $h^o_i\in P$,
 \item and, for every $j\in D$, a non-empty set $P^d(j)\subseteq P$ of any cardinality
\end{itemize}
to minimize
\begin{equation}
 \sum_{i\in O} \sum_{j\in D} \min_{m\in P^d(j)} C_{ijh^o_im}.  \label{f.o2}
\end{equation}

\section{Applications for the $(1,p)$-AHLP} \label{applications}

The $(1,p)$-AHLP model is motivated by the operational reality of modern, multi-echelon logistics, where the consolidation phase (upstream) and the distribution phase (downstream) follow distinct and often conflicting objectives. This framework is essential for systems that prioritize cost minimization and routing simplicity in the upstream phase while simultaneously demanding maximum resilience and service quality in the downstream phase. The structure offers two critical advantages: (i) Upstream efficiency, as the single allocation for the first hub simplifies routing and minimizes initial consolidation costs; and (ii) Downstream resilience, as the multiple allocation for the second hub ensures timely, flexible, and reliable delivery pathways. 

These applications generally fall into two broad classes: Global supply chain networks and humanitarian supply chains.

\subsection{Global supply chain networks}

The $(1,p)$-AHLP structure is most clearly manifested in commercial logistics operations that prioritize cost control at the consolidation stage and speed or reliability at the final delivery stage.

\bigskip

\textbf{Commercial flow applications}

\bigskip

Multinational companies managing factory-to-customer flows utilize this scheme: factories/suppliers (origins) dispatch goods to a specific regional distribution center (first hub, single allocation) for cost-efficient consolidation, which then feeds centralized global distribution centers (second hub). Finally, retailers and customers (destinations) receive the goods, benefiting from multiple allocation to ensure reception from the most cost-effective distribution center. Similar ${(1,p)}$ structures are fundamental to air cargo companies and e-commerce fulfillment networks.

For example, \textit{Amazon’s} fulfillment network is a clear case where origins (sellers/warehouses) send inventory via single allocation to first hubs (regional fulfillment centers) for consolidation. These centers stock the second hubs (centralized mega-fulfillment centers). Customers’ delivery addresses (destinations) are multiply allocated to these mega-centers to guarantee faster service, often leveraging the dual distribution capabilities from multiple large hubs. Similarly, \textit{FedEx}, one of the world's largest courier delivery services, operates a hub-and-spoke network where packages move from origins (local pickup points) via single allocation to first hubs (regional sorting centers). These centers feed second hubs (centralized global superhubs, like Memphis), which utilize multiple allocation to ship to destinations (final delivery addresses). A package destined for London, for example, could be routed through Memphis, Indianapolis, or even Paris, depending on the most efficient path.

\bigskip

\textbf{Non-physical flow applications}

\bigskip

The model also extends to non-physical flows, such as those in telecommunications networks, where local data sources (origins) ensure efficient aggregation through single allocation to regional data centers (first hub), which feed centralized data hubs or cloud servers (second hub). End-users or applications (destinations) are multiply allocated for redundancy and improved service quality. This structure is also applicable to energy distribution networks, with power plants (origins) feeding regional substations (first hub) that connect to national distribution hubs (second hub) supplying consumers (destinations) via multiple paths.

\subsection{Humanitarian supply chains}

The ${(1,p)}$-AHLP structure is essential in humanitarian logistics, where the primary goal is ensuring operational resilience and comprehensive coverage in crisis zones. The destinations (affected populations) receive aid or services from multiple second hubs. This design is highly valuable because the multiple allocation ensures redundancy (allowing aid to continue if one hub fails), provides flexibility (permitting dynamic routing based on changing need), and guarantees comprehensive coverage (making sure assistance reaches even hard-to-reach areas). Meanwhile, the single-allocation structure at the first hub maintains efficient consolidation of bulk supplies at regional centers.

\bigskip

\textbf{Relief network examples}

\bigskip

This structure is explicitly seen in major international relief efforts. Organizations like \textit{The United Nations Humanitarian Response Depots} (UNHRD) and \textit{The International Federation of Red Cross} (IFRC) utilize a similar two-level hub model. In both cases, origins (donors/suppliers) send relief items via single allocation to first hubs (regional warehouses, e.g., Dubai, Accra, Kuala Lumpur) for efficient storage and consolidation. These hubs then distribute to second hubs (local distribution centers or partner organizations in affected areas) via multiple allocation. The affected populations (destinations) receive aid from multiple local points, maximizing the likelihood of timely assistance and ensuring widespread coverage, as demonstrated during the 2015 Nepal earthquake (UNHRD) and the 2020 Beirut explosion (IFRC). Similarly, \textit{Médecins Sans Frontières} (MSF) and \textit{Save the Children} employ this dual structure for medical supplies and relief aid, respectively, ensuring that the destinations (patients/families) are multiply allocated to the local second hubs (field hospitals or distribution centers) to improve treatment access and responsiveness, as observed during the Ebola outbreak or Cyclone Idai.

\section{Integer programming formulations for \hbox{$(r,s)$-AHLP}}

In this section, we present the mathematical formulations designed to model and solve the \hbox{$(r,s)$-AHLP}. We first define the non-linear model before deriving linear programming formulations.

\subsection{The quadratic formulation}

We start by defining the set of decision variables required for the quadratic formulation.
The following $y$-variables are used to select the hubs:
\begin{equation*}
y_k = \left\{\begin{aligned}
1 & \quad \text{if $k$ is selected to be in $P$}\\
0 & \quad \text{otherwise} \end{aligned}  \qquad \forall k\in H. \right.
\end{equation*}
The sets $P^o(i)$ and $P^d(i)$ are determined by means of the following variables:
\begin{equation*}
x^o_{ik}=\left\{\begin{aligned}
1 & \quad \text{if $k$ belongs to $P^o(i)$}\\
0 & \quad \text{otherwise}
\end{aligned}  \qquad \forall i\in O,k\in H, \right.
\end{equation*}
\begin{equation*}
x^d_{jm}=\left\{\begin{aligned}
1 & \quad \text{if $m$ belongs to $P^d(j)$}\\
0 & \quad \text{otherwise}
\end{aligned}  \qquad \forall j\in D,m\in H. \right.
\end{equation*}
Along with $y_k$, $x^o_{ik}$ and $x^d_{ik}$, we will use other two sets of binary variables. Given $i\in O$, $j\in D$,
\begin{equation*}
z^o_{ijk}=\left\{\begin{aligned}
 1 & \quad \text{for one $k\in P^o(i)$ such that $m$ exists with }\\
   & \quad \text{$(k,m)\in \arg\min_{\ell \in P^o(i),t \in P^d(j)} C_{ij\ell t}$}\\
 0 & \quad \text{otherwise}, \end{aligned}  \right.
\end{equation*}

\begin{equation*}
z^d_{ijm}=\left\{\begin{aligned}
 1 & \quad \text{for one $m\in P^d(j)$ such that $k$ exists with }\\
   & \quad \text{$(k,m)\in \arg\min_{\ell \in P^o(i),t \in P^d(j)} C_{ij\ell t}$}\\
 0 & \quad \text{otherwise.} \end{aligned}  \right.
\end{equation*}

Now we give the quadratic formulation:
\begin{subequations} \label{rsq}
\begin{align}
&&((r,s)\text{-FQ})\ \min \quad &
\displaystyle{\sum_{i\in O}\sum_{j\in D}\sum_{k\in H}\sum_{m\in H} C_{ijkm}z^o_{ijk}z^d_{ijm}}\\
&&\text{s.t. } & \displaystyle{\sum_{k\in H} y_k \le p} \label{y} \\
&&& x^o_{ik}\le y_k && \forall i\in O,k\in H \label{xo} \\
&&& x^d_{jm}\le y_m && \forall j\in D,m\in H \label{xd} \\
&&& \displaystyle{1\le \sum_{k\in H} x^o_{ik} \le r} && \forall i \in O \label{r} \\
&&& \displaystyle{1\le \sum_{m\in H} x^d_{jm} \le s} && \forall j \in D \label{s} \\
&&&\displaystyle{\sum_{k\in H} z^o_{ijk} = 1}&& \forall i\in O,j\in D \label{zo1} \\
&&&\displaystyle{\sum_{m\in H} z^d_{ijm} = 1}&& \forall i\in O,j\in D \label{zd1} \\
&&&\displaystyle{z^o_{ijk} \le x^o_{ik}} && \forall i\in O,j\in D,k\in H \label{zxo} \\
&&&\displaystyle{z^d_{ijm} \le x^d_{jm}} && \forall i\in O,j\in D,m\in H \label{zxd} \\
&&& y_k\in \{ 0,1\} && \forall k \in H \label{y01} \\
&&& x^o_{ik}\in \{ 0,1\} && \forall i\in O,k \in H \label{xo01} \\
&&& x^d_{ik}\in \{ 0,1\} && \forall i\in D,k \in H \label{xd01} \\
&&& z^o_{ijk},z^d_{ijk}\in \{ 0,1\} && \forall i\in O,j \in D,k \in H. \label{z01}
\end{align}
\end{subequations}

The quadratic term is in the objective function. The number of hubs to be located is
set by \eqref{y}. Constraints  \eqref{xo} ensure that an origin $i$  is not assigned to a site $k$ unless $k$ is in the set of hubs $P$. The same is ensured with constraints  \eqref{xd} for the destinations.
Constraints \eqref{r} and \eqref{s}  guarantee that each origin and destination are assigned to at least one and at most $r$ (for the origin) and $s$ (for the destination) elements of $H$, respectively. Constraints \eqref{zo1} say that in every pair origin-destination $(i,j)$
there is a hub to which $i$ is assigned. The same is said for the destination $j$ in constraints \eqref{zd1}. 
Constraints \eqref{zxo} and \eqref{zxd} establish the relationship between $x$- and $z$-variables for origins and destinations.
The binary condition of variables is imposed from
\eqref{y01}-\eqref{z01}.

Note that the QSAP can be obtained from the $(r,s)$-AHLP by
(i) taking $O=D=H$,
(ii) taking $r=s=1$ (single allocation in both sides, then implying $z^o_{ijk}=x^o_{ik}$  and
$z^d_{ijm}=x^d_{jm}$),
(iii) taking $p=h$ (all elements in $H$ are hubs) and 
(iv) adding constraints $x^o_{ik}=x^d_{ik}$ $\forall i\in O$, $k\in H$ (allocation in both sides to the same hub).
In such a way, each site (origin and destination at the same time) will be assigned to exactly one site and each pair
of assignments will carry a given $C$-cost. On the other hand, the QAP requires also each hub to be assigned to only
one site, what can be achieved with the addition of constraints $\sum_{i\in O} x^o_{ik}=1$ $\forall k\in H$.

\subsection{Linear formulation based on variables with four indices}

The first linearization of formulation ($(r,s)$-FQ)  is an adaptation of the formulation for the $r$-AHLP originally proposed in \textcite{Yaman2011}. This model relies heavily on four-index variables ($X_{ijkm}$), which explicitly capture the entire path of the flow from origin $i$ to destination $j$ via hubs $k$ and $m$. This results in a formulation with a strong linear relaxation but one whose size becomes practically prohibitive for large-scale instances due to the multiplicative growth of the $X$-variables.

Given $i\in O$ and $j\in D$,
\begin{equation*}
X_{ijkm}=\left\{\begin{aligned}
 1 & \quad \text{for one $k\in P^o(i)$ and one $m\in P^d(j)$ }\\
   & \quad \text{such that $(k,m)\in \arg\min_{\ell \in P^o(i),t \in P^d(j)} C_{ij\ell t}$}\\
 0 & \quad \text{otherwise}. \end{aligned}  \right.
\end{equation*}

The formulation is
\begin{subequations}
\begin{align}
&&((r,s)\text{-F4)\ }\min \quad &
\displaystyle{\sum_{i\in O}\sum_{j\in D}\sum_{k\in H}\sum_{m\in H} C_{ijkm}X_{ijkm}}\\
&&\text{s.t. } &  \eqref{y},\eqref{xo},\eqref{xd},\eqref{r},\eqref{s},
                  \eqref{y01},\eqref{xo01},\eqref{xd01}\nonumber \\
&&& \displaystyle{\sum_{k\in H}\sum_{m\in H} X_{ijkm} = 1}&& \forall i\in O,j\in D \\
&&& \displaystyle{\sum_{k\in H} X_{ijkm} \le x^d_{jm}} && \forall i\in O,j\in D,m\in H \\
&&& \displaystyle{\sum_{m\in H} X_{ijkm} \le x^o_{ik}} && \forall i\in O,j\in D,k\in H\\
&&& X_{ijkm}\in \{ 0,1\} && \forall i\in O, j\in D, k,m \in H. \label{Xbin} 
\end{align}
\end{subequations}

Note that the number of variables is $O(o\cdot d\cdot h^2)$ and the number of constraints
is $O(o\cdot d\cdot h)$. 
%The following subsection shows that the linearization can be obtained with fewer variables and constraints.

\subsection{A new three-index formulation}

We develop a new formulation to linearize ($(r,s)$-FQ). This approach significantly reduces model size and is designed to overcome the computational limitations of the four-index model. This formulation reuses the flow variables ($z_{ijk}$) already defined in the quadratic model and introduces new auxiliary variables to linearize the objective function.

We introduce the new set of continuous variables, $\mu_{ij}$, that replace
in ($(r,s)$-F4) the value of $\sum_{k\in H}\sum_{m\in H} C_{ijkm}X_{ijkm}$. The formulation is then
\begin{subequations}
\begin{align}
&&((r,s)\text{-F3)\ }\min \quad &
\displaystyle{\sum_{i\in O} \sum_{j\in D} \mu_{ij}} \\
&&\text{s.t. } & \eqref{y},\eqref{xo},\eqref{xd},\eqref{r},\eqref{s},
                 \eqref{zo1},\eqref{zd1},\eqref{zxo},\eqref{zxd},
                 \eqref{y01},\eqref{xo01},\eqref{xd01},\eqref{z01}  \nonumber \\
&&& \displaystyle{\mu_{ij}\ge \sum_{m\in H} C_{ijkm} z^d_{ijm} -
    \sum_{\stackrel{\ell\in H}{\ell\neq k}} \max_{m\in H} (C_{ijkm}-C_{ij\ell m}) z^o_{ij\ell}}
&& \forall i\in O,j\in D,k\in H. \label{tela}
\end{align}
\end{subequations}

The meaning of constraints \eqref{tela} is the following. Consider any fixed values of
$i\in O$, $j\in D$ and $k\in H$. From \eqref{zo1} either $z^o_{ijk}=1$ or exactly
one $z^o_{ij\ell}$, $\ell\neq k$, will take value 1. If $z^o_{ijk}=1$, then \eqref{tela} will become
$\mu_{ij}\ge \sum_{m\in H} C_{ijkm} z^d_{ijm}$.
That is, if origin $i$ is assigned to hub $k$, then $\mu_{ij}$ is at least the overall cost
associated to the pair origin-destination $(i,j)$.
Since $((r,s)$-F3) is a minimization problem, $\mu_{ij}$ will be exactly this cost.
If, instead, $z^o_{ij\ell}=1$ for some $\ell\neq k$, \eqref{tela} will become
$$\mu_{ij}\ge \sum_{m\in H} C_{ijkm} z^d_{ijm} - \max_{m\in H} (C_{ijkm}-C_{ij\ell m}).$$
From \eqref{zd1} there will be exactly one $t\in D$ with $z^d_{ijt}=1$ and then the right-hand side of the previous inequality becomes
$$ C_{ijkt} - \max_{m\in H} (C_{ijkm}-C_{ij\ell m})=C_{ijkt} + \min_{m\in H} (-C_{ijkm}+C_{ij\ell m})\le C_{ijkt} -C_{ijkt}+C_{ij\ell t}=C_{ij\ell t},$$
which is a correct lower bound for $\mu_{ij}$ when $z^o_{ij\ell}=z^d_{ijt}=1$.

The strength of this formulation lies in its significant structural simplification, replacing the high-dimensional flow variables with lower-indexed ones. Although the asymptotic order of the total number of variables may be comparable to the four-index model, the practical reduction in the number of critical path-defining variables results in a less dense constraint matrix, leading to superior computational tractability.
	
\section{Solving the \hbox{$(1,p)$-AHLP}}

As motivated in the Introduction, we strategically focus our methodological development on the $(1,p)$-AHLP variant, given its high relevance in logistical scenarios that demand a critical balance between cost control (single origin assignment) and network resilience (multiple destination assignment).

\subsection{Formulations}

Given the allocation structure of this version of the problem, the assignment variables simplify significantly.

\begin{itemize}
\item Origin assignment: The $x$-variables are only required for the origins. Since $r=1$, we use the binary variable $x_{ik}$ to denote whether origin $i$ is assigned to hub $k$ (no need for a superscript $o$),
\begin{equation*}
x_{ik}=\left\{\begin{aligned}
1 & \quad \text{if origin $i$ is assigned to hub $k$} \\
0 & \quad \text{otherwise}
\end{aligned}  \qquad \forall i\in O,k\in H \right.
\end{equation*}
\item Destination flow/assignment: The $z$-variables are simplified as destinations are assigned to any open hub $m \in P$ (no need for a superscript $d$),
and taking into account that $P^o(i)=\{h^o_i\}$ $\forall i\in O$ and $P^d(j)=P$
$\forall j\in D$), then
\begin{equation}\label{eq_z}
z_{ijm}=\left\{\begin{aligned}
1 & \quad \text{for one $m\in P$ such that $(h^o_i,m)\in \arg\min_{t \in P} C_{ijh^o_it}$}\\
0 & \quad \text{otherwise}
\end{aligned}  \qquad \forall i\in O,j\in D,m\in H. \right.
\end{equation}
\item Hub location: We still use $y_k$ as defined previously.
\end{itemize}

The quadratic integer formulation for $(1,p)$-AHLP is the following: 
\begin{subequations}
\begin{align}
&&((1,p)\text{-FQ)\ }\min \quad &
\displaystyle{\sum_{i\in O} \sum_{j\in D}\sum_{k\in H}\sum_{m\in H} C_{ijkm}x_{ik}z_{ijm}}\\
&&\text{s.t. } & \sum_{k\in H} y_k \le p \label{otray} \\
&&& x_{ik}\le y_k && \forall i\in O,k\in H \label{x} \\
&&& \displaystyle{\sum_{k\in H} x_{ik} = 1} && \forall i \in O \label{x1} \\
&&& \displaystyle{\sum_{i\in O} \sum_{j\in D} z_{ijm}\le o\cdot d\cdot y_m} && \forall m\in H \label{zy} \\
&&& \displaystyle{\sum_{m\in H} z_{ijm} = 1}&& \forall i\in O,j\in D \label{z1} \\
&&& y_k\in \{ 0,1\} && \forall k \in H \label{yk01}\\
&&& x_{ik}\in \{ 0,1\} && \forall i\in O,k \in H \label{x01} \\
&&& z_{ijm}\in \{ 0,1\} && \forall i\in O,j\in D,m \in H. \label{zz01}
\end{align}
\end{subequations}

In order to develop a linear formulation suitable for Branch-and-Cut based resolution,
we introduce now a family of continuous variables
\begin{equation*}
\pi_i=\sum_{j\in D} \sum_{k\in H}\sum_{m\in H} C_{ijkm} x_{ik} z_{ijm}.
\end{equation*}
The difficulty
of the formulation is to obtain the values $\pi_i$ from the rest of the variables, in a
linear way. This is done in constraints \eqref{idea}, a family of size $o\cdot h$.
Previously we fix $i\in O$, $j\in D$ and $k\in H$ and sort, in increasing order, $C_{ijkm}$ for all $m\in H$. The
$(h-p+1)$-th entry of this sorted vector is denoted by $M_{ijk}$. Then, the
linear formulation is

\begin{subequations}
\begin{align}
&&((1,p)\text{-F)\ } & \min &&\displaystyle{\sum_{i\in O} \pi_i} \label{fo_pi}\\
&&& \text{s.t.} && \eqref{otray},\eqref{x},\eqref{x1},\eqref{zy},
                       \eqref{z1},\eqref{yk01},\eqref{x01},\eqref{zz01} 	\nonumber \\
&&&&&\pi_i\ge \sum_{j\in D} \bigg(
\sum_{m\in H}  C_{ijkm} z_{ijm} + M_{ijk}(x_{ik}-1)   \bigg)
 && \forall i\in O,k\in H \label{idea} \\
&&&&& \pi_i \ge 0 && \forall i\in O.
\end{align}
\end{subequations}

The meaning of constraints \eqref{idea} is the following. If $x_{ik}=1$, \eqref{idea}
will become $\pi_i\ge \sum_{j\in D} \sum_{m\in H} C_{ijkm} z_{ijm}$.
That is, if origin $i$ is assigned to hub $k$, then $\pi_i$ is at least the overall cost associated to origin $i$.
Since ($(1,p)$-F) is a minimization problem, $\pi_i$ will be exactly this cost.
If, instead, $x_{ik}=0$, \eqref{idea} becomes
$\pi_i\ge \sum_{j\in D}\sum_{m\in H} ( C_{ijkm} z_{ijm} - M_{ijk})$.
Taking into account that there will be $p$ available hubs to assign each destination $j$,
$\sum_{m\in H} C_{ijkm} z_{ijm}$ will be always in the optimum less than or equal to $M_{ijk}$.
Thus, $\pi_i$ will be lowerly bounded by at most 0.

As seen in the next result, formulation ((1,$p$)-F) can provide poor lower bounds.

\begin{propo} \label{cotacero}
The optimal value of the linear relaxation of {\rm ($(1,p)$-F)} can be 0.
\end{propo}
{\sc Proof.}
Consider a fractional solution where all $y$-, $x$- and $z$-variables take value $1/h$. Then constraints
\eqref{otray}, \eqref{x}, \eqref{x1} and \eqref{zy} are trivially satisfied.
Fixing values for $i$ and $k$, replacing the fractional solution in the right hand side of the associated
constraint \eqref{idea}, multiplying by $h$ and removing the useless indices it follows
$$\sum_{j\in D} (\sum_{m\in H} C_{jm} + M_j(1-h)).$$
Fixing and removing $j$, each addend is $\sum_{m\in H} (C_m-M) + M$ where
$M$ is the $(h-p+1)$-th largest value of the $h$-dimensional vector $(C_m)$. There are $h-p$ negative and
$p-1$ positive addends. So, there are cases where the sum is negative and the lower bound obtained with this formulation is 0, and other cases where the sum is positive and the lower bound will be greater than zero\vspace{0.2cm}.\hfill $\square$

Note that the use of the more natural but weaker value $M_{ijk}=\max_{m\in H} \{C_{ijkm}\}$ in constraints
\eqref{idea} will always lead to a lower bound equal to 0.

\subsection{Valid inequalities}

We have specifically sought a small formulation for $(1,p)$-AHLP, with the aim of
strengthening it by means of valid inequalities, while keeping the number of constraints
under control. For this reason, we added constraints \eqref{zy} and now introduce
the larger family of stronger constraints
\begin{equation}
z_{ijm}\le y_m \ \ \forall i\in O,\ j\in D,\ m\in H \label{zy2}
\end{equation}
as valid inequalities.
Note that the proof of Proposition \ref{cotacero} remains valid even when
constraints \eqref{zy} are replaced in ($(1,p)$-F) by constraints \eqref{zy2}.

The main family of valid inequalities follows from the following result.
\begin{lemma}
\textbf{Farkas' Lemma.} Let $A$ be an $n\times m$ dimensional matrix.
A vector $b\in \mathbb{R}^n$ verifies $Ax\leq b$ for $x\in \mathbb{R}^n$
with $x\ge \mathbf{0}$ if and only if for any $y\ge \mathbf{0}$ verifying
$A^Ty\ge \mathbf{0}$, it holds $b^Ty\ge 0$.
\end{lemma}

To apply this lemma, we consider three families of valid inequalities that relate $\pi_i$
and the $X$-variables of ($(1,p)$-F4):
\begin{subequations} \label{XXX}
\begin{align}
&&& \displaystyle{\sum_{m\in H} X_{ijkm} = x_{ik}} && \forall i\in O,j\in D,k\in H \\
&&& \displaystyle{\sum_{k\in H} X_{ijkm} = z_{ijm}} && \forall i\in O,j\in D,m\in H \\
&&& \pi_i \ge \sum_{j\in D} \sum_{k\in H}\sum_{m\in H} C_{ijkm} X_{ijkm} && \forall i\in O.
\end{align}
\end{subequations}

Fixing a value of $i\in O$ and using Farkas' Lemma with \eqref{XXX} by means of multipliers
$e_{jk}$ and $f_{jm}$, we directly obtain the following result:
\begin{propo}
Inequalities
\begin{eqnarray} \label{cortes}
\displaystyle \pi_i \ge
\sum_{j\in D} \sum_{k\in H} e_{jk} x_{ik} + \sum_{j\in D} \sum_{m\in H} f_{jm} z_{ijm}
\end{eqnarray}
for each $i\in O$ and $e_{ik}\in \mathbb{R}$, $f_{jm}\in \mathbb{R}$ such that
\begin{equation}\label{eq_ef}
e_{jk} + f_{jm} \le C_{ijkm}, \quad \forall j\in D,\ k,m\in H
\end{equation}
are valid for formulation {\rm ($(1,p)$-F)} .
\end{propo}

The most violated inequalities in \eqref{cortes} can be obtained as follows.
Consider $(\bar{y},\bar{x}, \bar{z}, \bar{\pi})$ a fractional solution to ($(1,p)$-F).
For every $i\in O$
we have to maximize the right-hand side of \eqref{cortes}  subject to condition \eqref{eq_ef}. This leads to the following auxiliary problem for each origin $i \in O$:
\begin{align}
\nonumber && \max \quad &
\displaystyle \sum_{j\in D} \sum_{k\in H} e_{jk} \bar{x}_{ik} +
              \sum_{j\in D} \sum_{m\in H} f_{jm} \bar{z}_{ijm} \\
\nonumber && \text{s.t.} \quad & e_{jk} + f_{jm} \le  C_{ijkm} &
  \forall j\in D,k,m\in H\\
  \nonumber &&& e_{jk}, f_{jm} \text{ unrestricted} &
  \forall j\in D,k,m\in H.
\end{align}
Each of these auxiliary problems can be split into $d$ independent subproblems, one for each destination $j\in D$. So, for fixed $i\in O$ and $j\in D$, we have to solve 
\begin{align}
\nonumber && \max & \sum_{k\in H} e_k \bar{x}_{ik} + 
                          \sum_{m \in H} f_m \bar{z}_{ijm} \\
\nonumber && \text{s.t. } & e_k + f_m \le C_{ijkm} & \forall k,m\in H \\
\nonumber &&& e_k, f_m \text{ unrestricted} & \forall k,m\in H.
\end{align}
This is the dual of a transportation problem and can
be efficiently solved. Its structure can be described as follows: The origins/destinations of the transportation problem are the potential hub locations $k \in H$ and $m \in H$, respectively. 
Supply at origin $k$ is $\bar{x}_{ik}$, representing the fractional allocation of origin $i$ to hub $k$. Demand at destination $m$ is $\bar{z}_{ijm}$, representing the fractional assignment of destination $j$ to hub $m$ associated with origin $i$. Transportation cost from  $k$ to  $m$ is $C_{ijkm}$.

\subsection{Particularization to disaggregated costs}

In this section we give the formulation of the problem for the particular case of disaggregated costs. For all $i\in O$ and $j\in D$, let $C_{ijkm}:=w_{ij}(\gamma c_{ik}+\alpha c_{km}+\beta c_{mj})$, with 
$0\le \alpha\le \beta,\gamma$. Thus, the objective function of our problem is 
\[
\sum_{i \in O} \sum_{j \in D} \sum_{k \in H} \gamma w_{ij} c_{ik} x_{ik} + 
\sum_{i \in O} \sum_{j \in D} \sum_{k \in H} \sum_{m \in H} 
w_{ij} (\alpha c_{km} + \beta c_{mj}) x_{ik} z_{ijm}.
\]

Now, for all $i\in O$ we define variables $\delta_i=\sum_{j\in D} \sum_{k\in H}\sum_{m\in H} w_{ij}(\alpha c_{km}+\beta c_{mj})x_{ik} z_{ijm}$.
The formulation ($(1,p)$-F) with disaggregated costs is

\begin{subequations}
\begin{align}
&&((1,p)\text{-FD)\ } & \min \displaystyle{
\sum_{i\in O} (\delta_i} + \sum_{j\in D} \sum_{k\in H} \gamma w_{ij} c_{ik} x_{ik})\\
&& \text{s.t.} \quad & \eqref{otray},\eqref{x},\eqref{x1},\eqref{zy},
                       \eqref{z1},\eqref{x01},\eqref{zz01} 	\nonumber \\
&&&\delta_i\ge \sum_{j\in D} \bigg(
\sum_{m\in H}w_{ij} (\alpha c_{km}+\beta c_{mj}) z_{ijm} + M'_{ijk}(x_{ik}-1) \bigg)
 && \forall i\in O,k\in H \label{idead} \\
 &&& \delta_i \ge 0 && \forall i\in O
\end{align}
\end{subequations}
where $M'_{ijk}$ is obtained sorting values $w_{ij}(\alpha c_{km}+\beta c_{mj})$ for all
$m\in H$ and taking the $(n-p+1)$-th of them.

The cuts given by (\ref{cortes}) can be adapted to this case:
\begin{eqnarray}\label{si_c2}
\displaystyle \delta_i \ge
\sum_{j\in D} \sum_{k\in H} e^*_{jk} x_{ik} + \sum_{j\in D} \sum_{m\in H } f^*_{jm}z_{ijm}, \quad \forall i \in O,
\end{eqnarray}
where (\mbox{\boldmath$e^*, f^*$}) denote the optimal solutions to the corresponding dual optimization models, i.e., are  obtained by solving,  
for all $i\in O$ and $j\in D$,
\begin{align*}
 (D_{ij})&&\max\ &
{\displaystyle \sum_{k\in H} e_{jk} \bar{x}_{ik} + \sum_{m\in H} f_{jm} \bar{z}_{ijm}} \\
\nonumber && \text{s.t.}\ & e_{jk} + f_{jm} \le w_{ij}(\alpha c_{km}+\beta c_{mj}) &
\forall j\in D,k\in H,m\in H\\
  \nonumber &&& e_{jk}, f_{jm} \text{ unrestricted}  &
\forall j\in D,k\in H,m\in H.
\end{align*}

\section{Computational study}

This section presents a computational study designed to evaluate the performance of the proposed formulations and the effectiveness of the exact solution procedure for the $(1,p)$-AHLP. 

We use the data files of the type SApHLP in the set AP (Australian Post) obtained from 
the OR-library \parencite{ORlib}. 
Computational experiments are performed on an Intel Xeon W-2245 CPU of 3.90 GHz with the Microsoft Windows 10 operating system. All formulations and algorithms are implemented in Xpress v8.10. Solver-specific settings, such as automatic cut generation and presolving, are typically disabled or carefully configured to isolate the impact of our proposed methodology. A CPU time limit of 7200 seconds is imposed for each instance.

\subsection{Performance analysis of the solution method}

We assess the efficiency of formulation ($(1,p)$-FD)
--given the disaggregated nature of the data-- 
and the strengthening impact of the valid inequalities.
We compare the following approaches. 
\begin {enumerate}
\item ($(1,p)$-FD), from now on denoted by FD.
\item FD with Valid Inequalities \eqref{zy2}, via separation.
\item FD with Valid Inequalities \eqref{si_c2}, via separation. The efficiency of generating these cuts is a key focus of our study, and we explore two methods for solving subproblems $(D_{ij})$: 
\begin{itemize}
\item Invoking Xpress also to solve the subproblems during the branch-and-cut process. This method is denoted by FD+\eqref{si_c2}.
\item Using a custom-implemented algorithm, specifically a primal–dual algorithm with preprocessing as described in \textcite{Haddadi2012}. To denote this, a superscript “T” will be added to constraints \eqref{si_c2}.
\end{itemize}
\item FD with both families of valid inequalities,  \eqref{si_c2} (applying
 the second strategy) and \eqref{zy2}. 
\end{enumerate}

The strategy to manage the custom cutting plane procedure, is as follows. The cut generation process is limited to a maximum of 5 iterations, and cut insertion is restricted to 
the root node of the branch-and-bound tree. A minimum cut violation tolerance of 0.01 is required for any inequality to be added.  
Within the same iteration, we apply a sequential strategy. We first focus exclusively on adding constraints of type \eqref{zy2}: In each round, all currently violated inequalities of this family are generated and added, as long as at least 100 violations are found. Once a round produces fewer than 100 such violations, we stop adding type \eqref{zy2} cuts in subsequent rounds and proceed to introduce cuts of type \eqref{si_c2}.

Table \ref{tab:apdata1} shows the numerical results for instances of size $n=40$ and $n=50$. The first column gives the values of $p$. The second column reports the procedure used to solve the problem. 
The following three columns provide the lower bound obtained after adding the cuts, if applicable (\emph{ILB}), the final lower bound (\emph{FLB}) and the final upper bound (\emph{FUP}) when time limit is reached. When an instance is solved within this time, this upper bound gives the optimal value.  
The total time (in seconds) to obtain the optimal solution is reported in column {\em Time}. When this time is reached without obtaining the optimal solution, \emph{TL} is indicated. Column {\em Gap} represents the gap (percent) between the best solution found and the best lower bound.
The last three columns show the number of nodes
explored in the branching tree and the number of constraints added as cuts of types \eqref{si_c2}  and  \eqref{zy2}, respectively. The symbol '-' indicates that it is not applicable.

\begin{table}[htbp]
\setlength\tabcolsep{0.3cm} \centering \scriptsize
    \begin{tabular}{llrrrrrrrr}
    \toprule
    $p$     & Form  & \multicolumn{1}{c}{ILB} & \multicolumn{1}{c}{FLB} & \multicolumn{1}{c}{FUB} & \multicolumn{1}{c}{Time} & \multicolumn{1}{c}{Gap} & \multicolumn{1}{c}{\# Nodes} & \multicolumn{1}{c}{\# \eqref{si_c2}} & \multicolumn{1}{c}{\# \eqref{zy2}} \\
    \midrule
     \multicolumn{10}{c}{$n=40$} \\
    \midrule
    \multirow{5}[2]{*}{2} &FD                                 & 86.82  & 174.78  &174.78  & 1092.4        & 0.00  & 452842 & -     & -    \\
                          &FD+\eqref{zy2}                     & 174.75 & 174.78  &174.78  & 1570.4        & 0.00  & 418032 & -     & 1924 \\
                          &FD+\eqref{si_c2}                   & 170.83 & 174.78  &174.78  & 226.1         & 0.00  & 2189   & 125   & -     \\
                          &FD+\eqref{si_c2}$^{T}$             & 168.91 & 174.78  &174.78  & \textbf{49.4} & 0.00  & 3853   & 89    & -     \\
                          &FD+\eqref{zy2}+\eqref{si_c2}$^{T}$ & 174.56 & 174.78  &174.78  & 143.6         & 0.00  & 90     & 88    & 10965 \\
    \midrule                                                                             
    \multirow{5}[2]{*}{3} &FD                                 &  74.99 & 113.7   &159.55  & TL            & 28.74 & 492222 & -     & -\\
                          &FD+\eqref{zy2}                     & 117.45 & 117.46  &161.72  & TL            & 27.37 & 483970 & -     & 2476\\
                          &FD+\eqref{si_c2}                   & 129.86 & 157.01  &157.01  & 3204.3        & 0.00  & 9097   & 122   & - \\
                          &FD+\eqref{si_c2}$^{T}$             & 150.08 & 157.01  &157.01  & 285.3         & 0.00  & 7990   & 81    & -\\
                          &FD+\eqref{zy2}+\eqref{si_c2}$^{T}$ & 156.76 & 157.01  &157.01  &\textbf{218.3} & 0.00  &847     & 92    & 13736 \\
    \midrule                                                                             
    \multirow{5}[2]{*}{4} &FD                                 & 64.58  & 87.97   &159.12  & TL            & 44.71 & 254030 & -     & -\\
                          &FD+\eqref{zy2}                     & 88.37  & 88.37   &151.81  & TL            & 41.79 & 113456 & -     & 2416 \\
                          &FD+\eqref{si_c2}                   & 124.73 & 145.68  &145.68  & 1589.2        & 0.00  & 28440  & 122   & - \\
                          &FD+\eqref{si_c2}$^{T}$             & 136.55 & 142.27  &142.27  & 254.7         & 0.00  & 18638  & 80    & - \\
                          &FD+\eqref{zy2}+\eqref{si_c2}$^{T}$ & 142.25 & 142.26  &142.27  & \textbf{173.4}& 0.01  &717     & 86    & 11190 \\
    \midrule                                                                             
    \multirow{5}[2]{*}{5} &FD                                 & 57.31  & 81.08   &143.77  & TL            & 43.60 & 171200 & -     & -\\
                          &FD+\eqref{zy2}                     & 74.9   & 74.9    &143.2   & TL            & 47.70 & 177442 & -     & 2748 \\
                          &FD+\eqref{si_c2}                   & 131.35 & 131.58  &131.58  & 3147.1        & 0.00  & 61833  & 105   & - \\
                          &FD+\eqref{si_c2}$^{T}$             & 131.45 & 131.58  &131.58  & 2462.9        & 0.00  & 76935  & 82    & -\\
                          &FD+\eqref{zy2}+\eqref{si_c2}$^{T}$ & 131.44 & 131.58  &131.58  & \textbf{125.2}& 0.00  & 537    & 90    & 10422 \\
    \midrule                                                                             
    \multirow{5}[2]{*}{6} &FD                                 & 51.46  & 74.64   &137.67  & TL            & 45.78 & 302947 & -     & -\\
                          &FD+\eqref{zy2}                     & 75.15  & 75.15   &139.65  & TL            & 46.19 & 173198 & -     & 2485 \\
                          &FD+\eqref{si_c2}                   & 123.06 & 123.53  &123.53  & 3053.2        & 0.00  & 247134 & 61    & - \\
                          &FD+\eqref{si_c2}$^{T}$             & 123.43 & 123.53  &123.53  & 1775.9        & 0.00  & 207688 & 76    & -\\
                          &FD+\eqref{zy2}+\eqref{si_c2}$^{T}$ & 123.38 & 123.53  &123.53  & \textbf{149.2}& 0.00  &1381    & 89    & 8704 \\
    \midrule                                                                             
    \multicolumn{10}{c}{$n=50$} \\                                                         
    \midrule                                                                             
    \multirow{5}[2]{*}{2} &FD                                 & 86.76  & 111.01  &176.3   & TL            & 37.03 & 107480 & -     & -\\
                          &FD+\eqref{zy2}                     & 133.49 & 133.49  &175.81  & TL            & 24.07 & 495194 & -     & 2733 \\
                          &FD+\eqref{si_c2}                   & 114.20 & 114.20  &175.95  & TL            & 35.10 & 0      & 50    & - \\
                          &FD+\eqref{si_c2}$^{T}$             & 173.05 & 175.79  &175.79  & \textbf{195.1}& 0.00  & 2745   &108    & - \\
                          &FD+\eqref{zy2}+\eqref{si_c2}$^{T}$ & 175.77 & 175.77  &175.79  & 1034.5        & 0.01  & 1560   & 96    & 22793 \\
    \midrule                                                                             
    \multirow{5}[2]{*}{3} &FD                                 & 74.77  & 84.52   &164.34  & TL            & 48.57 & 68345  & -     & -\\
                          &FD+\eqref{zy2}                     & 79.6   & 79.6    &171.79  & TL            & 53.66 & 35140  & -     & 3391 \\
                          &FD+\eqref{si_c2}                   & 153    & 156.9   &156.9   & 6340.8        & 0.00  & 12702  & 139   & - \\
                          &FD+\eqref{si_c2}$^{T}$             & 135.84 & 156.9   &156.9   & 773.4         & 0.00  & 15429  & 100   & - \\
                          &FD+\eqref{zy2}+\eqref{si_c2}$^{T}$ & 156.82 & 156.9   &156.9   & \textbf{728.3}& 0.00  & 733    & 105   & 25042 \\
    \midrule                                                                             
    \multirow{5}[2]{*}{4} &FD                                 & 64.23  & 68.74   &154.3   & TL            & 55.45 & 17967 & -     & -\\
                          &FD+\eqref{zy2}                     & 69.76  & 69.76   &148.59  & TL            & 53.05 & 24544 & -     & 3478 \\
                          &FD+\eqref{si_c2}                   & 65.45  & 65.45   &149.22  & TL            & 56.30 & 0     & 50    & - \\
                          &FD+\eqref{si_c2}$^{T}$             & 139.62 & 141.84  &141.84  & 2673.8        & 0.00  & 47855 & 100   & -\\
                          &FD+\eqref{zy2}+\eqref{si_c2}$^{T}$ & 141.65 & 141.83  &141.84  & \textbf{729.6}& 0.01  & 2487  & 99    & 20796 \\
    \midrule                                                                             
    \multirow{5}[2]{*}{5} &FD                                 & 55.66  & 62.64   &137.69  & TL            & 54.51 & 43410 & -     & -\\
                          &FD+\eqref{zy2}                     & 59.86  & 59.86   &147.69  & TL            & 59.47 & 28684 & -     & 3928 \\
                          &FD+\eqref{si_c2}                   & 62.16  & 62.16   &136.37  & TL            & 54.42 & 0     & 50    & - \\
                          &FD+\eqref{si_c2}$^{T}$             & 130.25 & 130.32  &130.32  & 5240          & 0.00  & 147045& 96    & -\\
                          &FD+\eqref{zy2}+\eqref{si_c2}$^{T}$ & 130.29 & 130.32  &130.32  & \textbf{753.7}& 0.00  & 2917  & 110   & 17850 \\
    \midrule                                                                             
    \multirow{5}[2]{*}{6} &FD                                 & 49.86  & 59.08   &141.48  & TL            & 58.24 & 80847 & -     & -\\
                          &FD+\eqref{zy2}                     & 57.9   & 57.9    &131.54  & TL            & 55.98 & 33965 & -     & 3572 \\
                          &FD+\eqref{si_c2}                   & 83.33  & 83.33   &83.33   & TL            & 0.00  & 0     & 86    & - \\
                          &FD+\eqref{si_c2}$^{T}$             & 113.63 & 113.63  &122.45  & TL            & 7.20  & 149589& 96    & -\\
                          &FD+\eqref{zy2}+\eqref{si_c2}$^{T}$ & 122.14 & 122.4   &122.4   & \textbf{801.8}& 0.00  & 4547  & 103   & 16673 \\
       \bottomrule
    \end{tabular}%
  \caption{Computational results for AP data set with $n=40$ and $n=50$ \label{tab:apdata1}}
\end{table}%

Computational results show that the formulation (FD) performs poorly in terms of computational time. This challenge is primarily attributed to the weak lower bound generated by its initial linear programming relaxation, indicating the necessity for strengthening techniques.
The addition of cuts from family \eqref{zy2} significantly improves the initial lower bound, as predicted; however, this bounding strength does not translate into a reduction of the overall solution times, leading to a very high computational cost.
The cuts of family \eqref{si_c2} prove to be highly effective in accelerating the solution process. Notably, our custom-implemented transportation problem algorithm (FD+\eqref{si_c2}$^T$) provides 
competitive solution times compared to the use of the solver  (FD+\eqref{si_c2}). 
The combined impact of incorporating both families of valid inequalities, \eqref{zy2} and \eqref{si_c2}, into the branch-and-cut framework achieves the most competitive solution times and lower bounds.

% Table generated by Excel2LaTeX from sheet 'TABLA'
\begin{table}[htbp]
  \centering
    \begin{tabular}{ccccc} \toprule
   $n$ & $p$       &  \multicolumn{1}{c}{hubs$_{\neq}$} & \multicolumn{1}{c}{Single Alloc.} & \multicolumn{1}{l}{Multiple Alloc.} \\ \toprule
    20&2 & 0     & 1     & 6 \\
    20&3 & 0     & 1     & 5 \\
    20&4 & 0     & 1     & 6 \\
    20&5 & 0     & 1     & 6 \\
    20&6 & 1     & 5     & 7 \\ \midrule
    30&2 & 0     & 0     & 10 \\
    30&3 & 0     & 2     & 8 \\
    30&4 & 0     & 1     & 8 \\
    30&5 & 0     & 1     & 10 \\
    30&6 & 0     & 0     & 9 \\ \midrule
    40&2 & 0     & 1     & 11 \\
    40&3 & 1     & 12    & 12 \\
    40&4 & 1     & 12    & 13 \\
    40&5 & 2     & 18    & 15 \\
    40&6 & 2     & 16    & 15 \\
    \bottomrule
    \end{tabular}%
    \caption{Comparing single and $(1,p)$-allocation\label{tab:Singlevsas}}
\end{table}%

\subsection{Impact of asymmetric $(1,p)$-allocation on the solution}

This subsection presents a comparative analysis of the $(1,p)$-AHLP and the classical Single Allocation Hub Location Problem (SAHLP). For solving the SAHLP, we utilize the formulation given in \textcite{sahlp}. The primary goal of this comparative analysis is structural: we are interested in analyzing how the optimal hub locations and allocation patterns change between the two models to understand the impact of asymmetric allocation on network design. 

To provide a detailed structural interpretation of the solutions, we focus our analysis on a smaller set of problem instances. The results are summarized in Table \ref{tab:Singlevsas}.

The first two columns report the instance size $n$ and the number of hubs $p$, respectively. The \emph{hubs}$_{\neq}$ column records the number of different hub locations selected by the two models, indicating a change in the strategic location of facilities. Since origins maintain single allocation in both models, the \emph{Single Alloc.}\ column tracks the number of origins whose unique assigned hub changes between the two solutions. Finally, the \emph{Multiple Alloc.}\ column records the number of destinations in the $(1,p)$-AHLP solution that are assigned to more than one hub, illustrating the utilization of the model's core flexibility in the distribution phase.

The computational study reveals that the introduction of destination multiple allocation often prompts changes in the optimal network structure. The \emph{hubs}$_{\neq}$ column shows that in several instances, the optimal set of hubs changes (e.g., $n=40$, $p=5$ has 2 different hubs), suggesting the multiple allocation flexibility influences the strategic placement of hubs.

As Table \ref{tab:Singlevsas} shows, the assignment of origins is re-optimized even when the hub locations are unchanged,  meaning reallocation can occur for the same set of hubs.  The \emph{Single Alloc.}\ column demonstrates that even when the set of hubs that are opened is the same when solving both the SAHLP and the $(1,p)$-AHLP models (\emph{hubs}$_{\neq}= 0$), the assignment of origins to these hubs is re-optimized to better utilize the multiple distribution pathways. For example, for the instance with $n = 30$ and $p = 3$, the set of hubs that are opened is the same, but two out of the thirty origins are assigned to different hubs depending on whether the destination follows the single-allocation or the $p$-allocation scheme (\emph{Single Alloc.} = 2).

The \emph{Multiple Alloc.}\ values, showing up to 15 destinations linked to multiple hubs, confirm that having a greater number of hubs enhances the connectivity of the distribution network, which is the main driver of the observed cost savings.

\section{Conclusions}

Our work introduces the novel Asymmetric Hub Location Problem (AHLP), addressing a critical theoretical and practical gap in traditional hub location models that often assume symmetric allocation of origins and destinations to hubs. This new framework allows origins and destinations to connect to hubs under distinct allocation limits, a necessity in real-world scenarios such as humanitarian logistics and global supply chains (e.g., UN relief networks, e-commerce fulfillment, and air cargo operations). We adapt the cost structure to allow for disaggregated costs, making the model highly flexible and applicable to various data structures.

Focusing on the $(1,p)$-AHLP variant (single origin assignment, multiple destination assignment), we develop a three-index formulation and at a later stage we significantly strengthen it using valid inequalities and decomposition techniques, ensuring a highly effective solution methodology.
The success of this methodology is confirmed by the computational study performed on a benchmark dataset commonly used in hub location literature. These studies also confirm that the optimal designs derived from the ${(1,p)}$-AHLP are structurally distinct from those obtained with other classic allocation criteria.

Further research will focus on extending the model's capabilities to address highly complex logistics environments. This includes incorporating capacity constraints on the hubs to manage realistic throughput limits and developing other procedures to efficiently solve larger problem instances. Finally, adapting the framework using stochastic or robust optimization methods will allow the model to operate effectively in environments characterized by uncertain demand or fluctuating costs.

\section*{Acknowledgements}
% Inmaculada Espejo has been partially supported by the following research projects: PCI2024-155092-2 (International Programme DFG-AEI 2023), PID2024-156594NB-C22, and RED2022-134149-T, all funded by the Ministerio de Ciencia, Innovación y Universidades, Spain. 
% Additional support was provided by the Consejería de Universidad, Investigación e Innovación de la Junta de Andalucía through project FEDER-UCA-2024-A2-02 (2021-2027 ERDF Operational Programme).

% Mercedes Landete acknowledges the support of grant PID2021-122344NB-I00, funded by MICIU/AEI/ 10.13039/501100011033 and ERDF/EU (Spain), as well as grant CIPROM/2024/34, funded by the Conselleria de Educación, Cultura, Universidades y Empleo, Generalitat Valenciana.
    
% Marina Leal acknowledges the support of grants PID2020-114594GB-C21 and PID2021-122344NB-I00 funded by MCIN/ AEI/ 10.13039/501100011033, and CIGE/2024/57 funded by Generalitat Valenciana.

% Alfredo Marín acknowledges that this article is part of the research projects PID2022-137818OB-I00 (Ministerio
% de Ciencia e Innovaci\'on, Spain), RED2022-134149-T funded by MCIN/AEI/10.13039/501100011033. Additional support was provided by the Consejería de Universidad, Investigación e Innovación de la Junta de Andalucía through project FEDER-UCA-2024-A2-02 (2021-2027 ERDF Operational Programme).

%\addbibresource{references} %Import the bibliography file

%\bibliography{references.bib}
%\bibliographystyle{authoryear}
%\bibliography{references.bib}
\printbibliography %Prints bibliography

\end{document}